\documentclass[12pt]{article}
\usepackage{amsmath,amssymb,amsthm, comment}

\newtheorem{theorem}{Theorem}[section]
\newtheorem{thmy}{Theorem}

\newtheorem{lemma}[theorem]{Lemma}

\def\barr{\begin{array}}
\def\earr{\end{array}}

\title{Detecting structural properties of finite groups by the sum of element orders}
\author{Marius T\u arn\u auceanu}
\date{April 6, 2019}

\begin{document}

\maketitle

\begin{abstract}
In this paper, we introduce a new function related to the sum of element orders of finite groups. It is used
to give some criteria for a finite group to be cyclic, abelian, nilpotent, supersolvable and solvable, respectively.
\end{abstract}

{\small
\noindent
{\bf MSC2000\,:} Primary 20D60; Secondary 20D10, 20D15, 20F16, 20F18.

\noindent
{\bf Key words\,:} group element orders, cyclic groups, abelian groups, nilpotent groups, supersolvable groups, solvable groups.}

\section{Introduction}
Given a finite group $G$, we consider the function
\begin{equation}
\psi(G)=\sum_{x\in G}o(x),\nonumber
\end{equation}where $o(x)$ denotes the order of $x$. This has been introduced by H. Amiri, S.M. Jafarian Amiri and I.M. Isaacs \cite{1}. They proved the following theorem:

\begin{thmy}
If $G$ is a group of order $n$, then $\psi(G)\leq\psi(C_n)$, and we have equality if and only if $G$ is cyclic.
\end{thmy}In other words, the cyclic group $C_n$ is the unique group of order $n$ which attains the maximal value of $\psi(G)$ among groups of order $n$.\newpage

Since then many authors have studied the function $\psi(G)$ and its relations with the structure of $G$ (see e.g. \cite{2}-\cite{5}, \cite{7}-\cite{10}, \cite{12} and \cite{14}). In the papers \cite{4} and \cite{12} M. Amiri and S.M. Jafarian Amiri, and, independently, R. Shen, G. Chen and C. Wu started the investigation of groups with the second largest value of the sum of element orders. M. Herzog, P. Longobardi and M. Maj \cite{7} determined the exact upper bound for $\psi(G)$ for non-cyclic groups of order $n$:

\begin{thmy}
If $G$ is a non-cylic group of order $n$ and $q$ is the least prime divisor of the order of $n$, then
\begin{equation}
\psi(G)\leq\frac{\left[(q^2-1)q+1\right](q+1)}{q^5+1}\,\psi(C_n)=f(q)\psi(C_n).\nonumber
\end{equation}Moreover, the equality holds if and only if $n=q^2m$ with $(m,q!)=1$ and $G\cong(C_q\times C_q)\times C_m$.
\end{thmy}Note that the above function $f$ is strictly decreasing on $[2,\infty)$. Consequently, we have
\begin{equation}
\psi(G)\leq f(2)\psi(C_n)=\frac{7}{11}\,\psi(C_n),\nonumber
\end{equation}and the equality holds for $n=4m$ with $m$ odd and $G\cong(C_2\times C_2)\times C_m$.
\smallskip

By using the sum of element orders, several criteria for solvability of finite groups have been determined (see e.g. \cite{5,8}). We recall here the following theorem of M. Baniasad Asad and B. Khosravi \cite{5}:

\begin{thmy}
If $G$ is a group of order $n$ and $\psi(G)>\frac{211}{1617}\,\psi(C_n)$, then $G$ is solvable.
\end{thmy}Note that the equality $\psi(G)=\frac{211}{1617}\,\psi(C_n)$ occurs for $n=60m$ with $(30,m)=1$ and $G\cong A_5\times C_m$.
\smallskip

We also recall a criterion for nilpotency of finite groups that has been proved in \cite{14}:

\begin{thmy}
If $G$ is a group of order $n$ and $\psi(G)>\frac{13}{21}\,\psi(C_n)$, then $G$ is nilpotent. Moreover, we have $\psi(G)=\frac{13}{21}\,\psi(C_n)$ if and only if $n=6m$ with $(6,m)=1$ and $G\cong S_3\times C_m$.
\end{thmy}

The largest four values of the ratio $\psi'(G)=\frac{\psi(G)}{\psi(C_{|G|})}$ and the groups $G$ for which they are attained can be obtained from Theorem D.

\bigskip\noindent{\bf Corollary E.} {\it Let $G$ be a finite group satisfying $\psi'(G)>\frac{13}{21}$\,. Then $\psi'(G)\in\{\frac{27}{43}\,, \frac{7}{11}\,, 1\}$, and one of the following holds:
\begin{itemize}
\item[{\rm a)}] $G\cong Q_8\times C_m$, where $m$ is odd;
\item[{\rm b)}] $G\cong (C_2\times C_2)\times C_m$, where $m$ is odd;
\item[{\rm c)}] $G$ is cyclic.
\end{itemize}}

The above results show that a finite group $G$ becomes cyclic, abelian, nilpotent or solvable if $\psi'(G)$ is sufficiently large\footnote{A similar result for supersolvability has been conjectured in \cite{14}.}. In what follows, we consider the function
\begin{equation}
\psi''(G)=\frac{\psi(G)}{|G|^2}\,.\nonumber
\end{equation}Clearly, $\psi''(G)<1$ if $G$ is non-trivial, and there are sequences of groups $(G_n)$ such that $\psi''(G_n)$ tends to $1$ when $n$ tends to infinity (for example, $(C_p)$ where $p$ runs over the set of primes). We also observe that $\psi''$ satisfies the following important property
\begin{equation}
\psi''(G)\leq\psi''(G/H),\, \forall\, H\lhd G,
\end{equation}by Proposition 2.6 of \cite{8}. We will use this new function to give criteria for a finite group to be cyclic, abelian, nilpotent, supersolvable and solvable, respectively. Our main result is the following theorem.

\begin{theorem}
Let $G$ be a finite group. Then the following hold:
\begin{itemize}
\item[{\rm a)}] If $\psi''(G)>\frac{7}{16}=\psi''(C_2\times C_2)$, then $G$ is cyclic;
\item[{\rm b)}] If $\psi''(G)>\frac{27}{64}=\psi''(Q_8)$, then $G$ is abelian;
\item[{\rm c)}] If $\psi''(G)>\frac{13}{36}=\psi''(S_3)$, then $G$ is nilpotent;
\item[{\rm d)}] If $\psi''(G)>\frac{31}{144}=\psi''(A_4)$, then $G$ is supersolvable;
\item[{\rm e)}] If $\psi''(G)>\frac{211}{3600}=\psi''(A_5)$, then $G$ is solvable.
\end{itemize}
\end{theorem}
\smallskip

Note that the converses of the implications in Theorem 1.1 are not true. We observe that $\psi''(C_{p^m})$ tends to $\frac{p}{p+1}$ when $m$ tends to infinity. Let $(p_i)_{i\in\mathbb{N}^*}$ be the sequence of primes. Since $\prod_{i\geq 1}\frac{p_i}{p_i+1}=0$, there exists a positive integer $k$ such that $c=\prod_{i=1}^k \frac{p_i}{p_i+1}<\frac{211}{3600}\,$. If $n=\prod_{i=1}^k p_i^{n_i}$, then $\psi''(C_n)$ tends to $c$ when $n_1$, ..., $n_k$ tend to infinity. In other words, $\psi''(C_n)<\frac{211}{3600}$ for $n_1$, ..., $n_k$ sufficiently large, i.e. there are cyclic groups $G$ with $\psi''(G)<\frac{211}{3600}\,\vspace{2mm}$.\newpage

For the proof of the above theorem, we need some preliminary results about the function $\psi$ taken from \cite{2,7}.

\begin{lemma}
Here $G$ denotes a finite group, $p$, $p_i$ denote primes and $n$, $n_i$, denote positive integers. The following statements hold:
\begin{itemize}
\item[{\rm 1)}] {\rm (\cite{7}, Lemma 2.9(1))} $\psi(C_{p^n})=\frac{p^{2n+1}+1}{p+1}$\,;
\item[{\rm 2)}] {\rm (\cite{7}, Lemma 2.2(3))} $\psi$ is multiplicative, that is if $G=A\times B$, where $A,B$ are subgroups of $G$ satisfying $\gcd(|A|,|B|)=1$, then $\psi(G)=\psi(A)\psi(B)$;
\item[{\rm 3)}] {\rm (\cite{2}, Lemma 2.1)} $\psi(A\times B)\leq \psi(A)\psi(B)$. Moreover, $\psi(A\times B)=\psi(A)\psi(B)$ if and only if $\gcd(|A|,|B|)=1$;
\item[{\rm 4)}] {\rm (\cite{7}, Lemma 2.9(2))} If $n=\prod_{i=1}^k p_i^{n_i}$, where $p_i\neq p_j$ for $i\neq j$, then $\psi(C_n)=\prod_{i=1}^k \psi(C_{p_i^{n_i}})$;
\item[{\rm 5)}] {\rm (\cite{7}, Lemma 2.2(5))} If $G=P\rtimes H$, where $P$ is a cyclic $p$-group, $|H|>1$ and $(p,|H|)=1$, then $\psi(G)=|P|\psi(H)+(\psi(P)-|P|)\psi(C_H(P))$.
\end{itemize}
\end{lemma}

We also need the following two theorems. The first one is due to A. Lucchini (see Theorem 2.20 in \cite{6}), while the second one is a consequence of a theorem of B. Huppert and N. Ito (see Theorem 13.10.1 in \cite{11}).

\begin{theorem}
Let $A$ be a cyclic proper subgroup of a finite group $G$, and let $K={\rm Core}_G(A)$. Then
$[A:K]<[G:A]$, and in particular, if $|A|\geq[G:A]$, then $K>1$.
\end{theorem}

\begin{theorem}
Suppose that a finite group $G$ contains a subgroup $B$ of prime power index and $B$ contains a cyclic subgroup $H$ of index $[B:H]\leq 2$. Then $G$ is solvable.
\end{theorem}

We end our paper by indicating a natural open problem concerning the criteria in Theorem 1.1.

\bigskip\noindent{\bf Open problem.} Determine all finite groups $G$ for which $\psi''(G)$ takes the values $\frac{7}{16}\,$, $\frac{27}{64}\,$, $\frac{13}{36}\,$, $\frac{31}{144}$ and $\frac{211}{3600}\,$, respectively.
\bigskip

Note that, given $c\in(0,1)\cap\mathbb{Q}$, the main difficulty in solving this problem is to determine positive integers $n$ such that $\psi''(C_n)=c$.

\section{Proofs of the main results}

First of all, we give three lemmas that will be useful to us.

\begin{lemma}
Let $G$ be a finite group. If $\psi''(G)\geq\frac{1}{3}$, then either $G$ is cyclic or there exists $x\in G$ such that $[G:\langle x\rangle]=2$.
\end{lemma}

\begin{proof}
From the condition $\psi''(G)\geq\frac{1}{3}$ we infer that there exists $x\in G$ such that $o(x)>\frac{|G|}{3}\,$. This leads to $[G:\langle x\rangle]<3$, that is $[G:\langle x\rangle]\in\{1,2\}$. Then either $G=\langle x\rangle$ is cyclic or $[G:\langle x\rangle]=2$, as desired.
\end{proof}

\begin{lemma}
Let $G$ be a finite $2$-group having a cyclic maximal subgroup. Then the following hold:
\begin{itemize}
\item[{\rm a)}] If $\psi''(G)>\frac{7}{16}\,$, then $G$ is cyclic;
\item[{\rm b)}] If $\psi''(G)>\frac{27}{64}\,$, then $G$ is cyclic or $G\cong C_2\times C_2$;
\item[{\rm c)}] If $\psi''(G)>\frac{13}{36}\,$, then $G$ is cyclic or $G\cong C_2\times C_2$ or $G\cong Q_8$.
\end{itemize}
\end{lemma}

\begin{proof}
Assume that $G$ is not cyclic and let $n=|G|$. Then, by Theorem 4.1 of \cite{13}, II, we infer that either $G$ is abelian of type $C_2\times C_{2^{n-1}}$, $n\geq 2$, or non-abelian of one of the following types:
\begin{itemize}
\item[-] $M(2^{n})=\langle x,y\mid x^{2^{n-1}}=y^2=1\,, yxy=x^{2^{n-2}+1}\rangle$, $n\geq 4$;
\item[-] $D_{2^{n}}=\langle x,y\mid x^{2^{n-1}}=y^2=1\,, yxy=x^{-1}\rangle$, $n\geq 3$;
\item[-] $Q_{2^{n}}=\langle x,y\mid x^{2^{n-1}}=y^4=1\,, yxy^{-1}=x^{2^{n-1}-1}\rangle$, $n\geq 3$;
\item[-] $S_{2^{n}}=\langle x,y\mid x^{2^{n-1}}=y^2=1\,, yxy=x^{2^{n-2}-1}\rangle$, $n\geq 4$.
\end{itemize}If $G\cong C_2\times C_{2^{n-1}}$, then
\begin{equation}
\psi''(G)=\frac{2^{2n}+5}{3\cdot 2^{2n}}>\frac{13}{36}\Leftrightarrow 2^{2n}<60\Leftrightarrow n=2, \mbox{ i.e. } G\cong C_2\times C_2,\nonumber
\end{equation}while if $G$ is non-abelian, then we get:
\begin{itemize}
\item[-] $\psi''(M(2^n))=\displaystyle\frac{1}{2^n}<\frac{13}{36}\,,\,\forall\, n\geq 4$;
\item[-] $\psi''(D_{2^n})=\displaystyle\frac{2^{2n-1}+3\cdot 2^n+1}{3\cdot 2^{2n}}<\frac{13}{36}\,,\,\forall\, n\geq 3$;
\item[-] $\psi''(Q_{2^n})=\displaystyle\frac{2^{2n-1}+3\cdot 2^{n+1}+1}{3\cdot 2^{2n}}>\frac{13}{36}\Leftrightarrow n=3$, i.e. $G\cong Q_8$;
\item[-] $\psi''(S_{2^n})=\displaystyle\frac{2^{2n-1}+9\cdot 2^{n-1}+1}{3\cdot 2^{2n}}<\frac{13}{36}\,,\,\forall\, n\geq 4$.
\end{itemize}This completes the proof.
\end{proof}

\begin{lemma}
Let $G$ be a non-trivial semidirect product of $C_{p^n}$ by $C_2$, where $p$ is an odd prime and $n$ is a positive integer. Then $\psi''(G)\leq\frac{13}{36}\,$, and the equality occurs if and only if $p=3$ and $n=1$, i.e. $G\cong S_3$.
\end{lemma}

\begin{proof}
Lemma 1.2, 5), shows that
\begin{equation}
\psi(G)=\psi(C_{p^n}\rtimes C_2)=p^n\psi(C_2)+(\psi(C_{p^n})-p^n)\psi(C_{C_2}(C_{p^n})).\nonumber
\end{equation}Since the semidirect product $C_{p^n}\rtimes C_2$ is non-trivial, we have $C_{C_2}(C_{p^n})=1$ and so
\begin{equation}
\psi(G)=3\cdot p^n+\frac{p^{2n+1}+1}{p+1}-p^n=\frac{p^{2n+1}+2\cdot p^{n+1}+2\cdot p^{n}+1}{p+1}\,,\nonumber
\end{equation}i.e.
\begin{equation}
\psi''(G)=\frac{p^{2n+1}+2\cdot p^{n+1}+2\cdot p^{n}+1}{4(p^{2n+1}+p^{2n})}\,.\nonumber
\end{equation}Now, the inequality $\psi''(G)\leq\frac{13}{36}$ is equivalent with
\begin{equation}
4\cdot p^{2n+1}+13\cdot p^{2n}-18\cdot p^{n+1}-18\cdot p^n-9\geq 0.\nonumber
\end{equation}We easily observe that the following function on the real variable
\begin{equation}
f(x)=4\cdot x^{2n+1}+13\cdot x^{2n}-18\cdot x^{n+1}-18\cdot x^n-9\nonumber
\end{equation}is strictly increasing on $[3,\infty)$. Consequently,
\begin{equation}
f(x)\geq f(3)=25\cdot 3^{2n}{-}72\cdot 3^{n}{-}9\geq 25\cdot 3^{n+1}{-}72\cdot 3^{n}{-}9=3^{n+1}{-}9\geq 0\nonumber
\end{equation}and we have equality if and only if $x=3$ and $n=1$, completing the proof.
\end{proof}

We are now able to prove our main result.

\bigskip\noindent{\bf Proof of Theorem 1.1.} We first prove item c). Since $\psi''(G)\geq\frac{13}{36}\geq\frac{1}{3}\,$, by Lemma 2.1 it follows that either $G$ is cyclic or there exists $x\in G$ such that $[G:\langle x\rangle]=2$. Let $n=|G|=2^{n_1}p_2^{n_2}\cdots p_k^{n_k}$, where $p_2<p_3<\cdots<p_k$ are odd primes. Then, for each $i=2,...,k$, $\langle x\rangle$ contains a cyclic Sylow $p_i$-subgroup $P_i$ of $G$. We have $\langle x\rangle\leq N_G(P_i)$ and so $P_i$ is normal in $G$. Therefore $G$ has a cyclic normal $2$-complement, that is\newpage
\begin{equation}
G=C_m\rtimes H,\nonumber
\end{equation}where $m=\frac{n}{2^{n_1}}$ is odd and $H$ is a Sylow $2$-subgroup of $G$. Note that $\langle x\rangle$ also contains a cyclic normal subgroup $M$ of order $2^{n_1-1}$, that is $H$ possesses a cyclic maximal subgroup. Assume that $G$ is not nilpotent. We will show that $\psi''(G)\leq\frac{13}{36}\,$, contradicting our hypothesis. We infer that there exists $i\in\{2,...,k\}$ such that the semidirect product $C_{p_i^{n_i}}\rtimes H$ is non-trivial, and by property (1) we may assume that $G=C_{p_i^{n_i}}\rtimes H$. Then $G/M$ is a non-trivial semidirect product of $C_{p_i^{n_i}}$ by $C_2$, and (1) and Lemma 2.3 imply that
\begin{equation}
\psi''(G)\leq\psi''(\frac{G}{M})=\psi''(C_{p_i^{n_i}}\rtimes C_2)\leq\frac{13}{36}\,,\nonumber
\end{equation}as desired. Consequently, $G$ is nilpotent.

Next we prove items a) and b). If $\psi''(G)>\frac{7}{16}\,$, then $G$ is nilpotent by c). Under the above notations, we have $G=C_m\times H$ and
\begin{equation}
\psi''(H)=\psi''(\frac{G}{C_m})\geq\psi''(G)>\frac{7}{16}\nonumber
\end{equation}implies that $H$ is cyclic by Lemma 2.2. Thus $G$ is cyclic. Similarly, if $\psi''(G)>\frac{27}{64}\,$, then one obtains $G=C_m\times H$, where either $H$ is cyclic or $H\cong C_2\times C_2$. Thus $G$ is abelian.

We prove now item d). We proceed by induction on $|G|$. Since cyclic-by-supersolvable groups are supersolvable, it suffices to show that $G$ contains a non-trivial cyclic normal subgroup $C$. Indeed, in this case we would have
\begin{equation}
\psi''(\frac{G}{C})\geq\psi''(G)>\frac{31}{144}\nonumber
\end{equation}and so $G/C$ would be supersolvable by the inductive hypothesis. It is clear that the condition $\psi''(G)>\frac{31}{144}$ implies that there exists $x\in G$ such that
\begin{equation}
[G:\langle x\rangle]<\frac{144}{31}\,, \mbox{ i.e. } [G:\langle x\rangle]\in\{1,2,3,4\}.\nonumber
\end{equation}Obviously, we can choose $C=\langle x\rangle$ for $[G:\langle x\rangle]\in\{1,2\}$. Assume that $[G:\langle x\rangle]=3$. If ${\rm Core}_G(\langle x\rangle)\neq 1$, then we can choose $C={\rm Core}_G(\langle x\rangle)$, while if ${\rm Core}_G(\langle x\rangle)=1$ we infer that $G$ is supersolvable because it can be embedded in $S_3$. Assume now that $[G:\langle x\rangle]=4$. Again, if ${\rm Core}_G(\langle x\rangle)\neq 1$, then we can choose $C={\rm Core}_G(\langle x\rangle)$, while if ${\rm Core}_G(\langle x\rangle)=1$ we infer that $G$ can be embedded in $S_4$. Since $A_4$ and $S_4$ are the unique non-supersolvable subgroups of $S_4$ and
\begin{equation}
\psi''(S_4)=\frac{67}{576}<\psi''(A_4)=\frac{31}{144}<\psi''(G)\,,\nonumber
\end{equation}it follows that $G$ is supersolvable.

Finally, we prove item e). Similarly with d), it suffices to show that $G$ contains a non-trivial cyclic normal subgroup. The condition $\psi''(G)>\frac{211}{3600}$ implies that there exists $x\in G$ such that
\begin{equation}
[G:\langle x\rangle]\leq 17.\nonumber
\end{equation}If $[G:\langle x\rangle]\in\{16,17\}$, then the conclusion follows by Theorem 1.4. So, we can suppose that 
\begin{equation}
[G:\langle x\rangle]\leq 15.\nonumber
\end{equation} 
If $|G|\geq 225$, then
\begin{equation}
|\langle x\rangle|\geq\frac{225}{[G:\langle x\rangle]}\geq [G:\langle x\rangle]\nonumber
\end{equation}and therefore ${\rm Core}_G(\langle x\rangle)$ is a non-trivial cyclic normal subgroup of $G$ by Theorem 1.3. Suppose now that $|G|<225$ and that $G$ is non-solvable. Then one of the following holds:
\begin{itemize}
\item[(i)] $|G|=60$ and $G\cong A_5$;
\item[(ii)] $|G|=120$ and $G\cong A_5\times C_2$ or $G\cong S_5$ or $G\cong {\rm SL}(2,5)$;
\item[(iii)] $|G|=168$ and $G\cong {\rm PSL}(2,7)$;
\item[(iv)] $|G|=180$ and $G\cong A_5\times C_3$.
\end{itemize}If $G\cong A_5\times C_n$ with $n=2,3$, then Lemma 1.2, 3), leads to 
\begin{equation}
\psi''(G)\leq\psi''(A_5)\psi''(C_n)<\psi''(A_5),\nonumber
\end{equation}a contradiction. Using GAP, in the other cases we get
\begin{itemize}
\item[] \hspace{32mm} $\psi''(S_5)=\frac{471}{14400}<\frac{211}{3600}\,$,
\item[] \hspace{32mm} $\psi''({\rm SL}(2,5))=\frac{663}{14400}<\frac{211}{3600}\,$,
\item[] \hspace{32mm} $\psi''({\rm PSL}(2,7))=\frac{715}{28224}<\frac{211}{3600}\,$,
\end{itemize}contradicting again the hypothesis.

The proof of Theorem 1.1 is now complete.\qed

\vspace*{3ex}\small

\hfill
\begin{minipage}[t]{5cm}
Marius T\u arn\u auceanu \\
Faculty of  Mathematics \\
``Al.I. Cuza'' University \\
Ia\c si, Romania \\
e-mail: {\tt tarnauc@uaic.ro}
\end{minipage}

\end{document}